\documentclass[12pt]{article}

\usepackage{amsmath, amsthm, amssymb}

\usepackage{fullpage}

\usepackage{hyperref}

\newtheorem{theorem}{Theorem}

\theoremstyle{definition}
\newtheorem{question}{Question}

\theoremstyle{remark}

\title{Improved lower bounds on the extrema of eigenvalues of graphs}
\author{William Linz\thanks{University of South Carolina, Columbia, SC. ({\tt wlinz@mailbox.sc.edu}). Partially supported by NSF RTG Grant DMS 2038080}}
\date{\today}

\begin{document}
\maketitle
\begin{abstract}
In this note, we improve the lower bounds for the maximum size of the $k$th largest eigenvalue of the adjacency matrix of a graph for several values of $k$. In particular, we show that closed blowups of the icosahedral graph improve the lower bound for the maximum size of the fourth largest eigenvalue of a graph, answering a question of Nikiforov.
\end{abstract}

\section{Introduction}

How large can the $k$th largest eigenvalue of a graph $G$ on $n$ vertices be? The graphs $k K_{\frac{n}{k}}$ show that the $k$th largest eigenvalue can be at least $\frac{n}{k} - 1$ (we assume $n$ is a multiple of $k$ here for simplicity). Can this easy lower bound be improved? 

To fix notation, for a graph $G$ on $n$ vertices, we denote the eigenvalues of the adjacency matrix of $G$ by $\lambda_1 \ge \lambda_2 \ge \cdots \ge\lambda_n$. Following Nikiforov~\cite{Nik1}, we define $\lambda_k(n) = \max_{|V(G)| = n} \lambda_k(G)$ and $c_k = \sup\{\lambda_k(G)/n: |V(G)| = n, n\ge k\}$. In fact, Nikiforov shows $c_k = \lim_{n\rightarrow \infty} \lambda_{k}(n)/n$, by methods introduced in \cite{Nik2}. 

The question of providing good upper and lower bounds on the $k$th largest eigenvalue $\lambda_k$ of a graph was apparently first stated by Hong~\cite{Hon}. Nikiforov was able to prove the following bounds on $c_k$. 

\begin{theorem}[Nikiforov~\cite{Nik1}]\label{thm:nik}
Let $k \ge 2$. Then, 
\[c_k \le \frac{1}{2\sqrt{k-1}}.\]
Furthermore, there exists an integer $k_0$ such that for any $k > k_0$, 
\[c_k \ge \frac{1}{2\sqrt{k-1} + \sqrt[3]{k}}.\]
\end{theorem}

Nikiforov also showed that $c_k \ge \frac{1}{k - \frac12}$ for all $k \ge 5$, improving on the lower bound given by $kK_{\frac{n}{k}}$. On the other hand, $c_k = \frac{1}{k}$ for $k=1$ and $k=2$, leaving only the cases $k=3$ and $k=4$ open for the question in the beginning paragraph.

\begin{question}[Nikiforov~\cite{Nik1}]
Is $c_3 = \frac13$? Is $c_4 = \frac14$?
\end{question}

In this note, we answer half of Nikiforov's question, improving the lower bound on $c_4$. 

\begin{theorem}\label{thm:c4}
\[c_4 \ge \frac{1+\sqrt{5}}{12} \approx 0.26967.\]
\end{theorem}

We can also improve the best known lower bound on $c_k$ for many other small values of $k$.

\begin{theorem}\label{thm:ck}
For $6\le k\le 16$,
\[c_k \ge \frac{2(k-3)}{k(k-1)}.\]
\end{theorem}
The lower bound in Theorem~\ref{thm:ck} is in fact valid for all $k\ge 4$, but there are better bounds for $4\le k\le 5$ and $k\ge 17$. Furthermore, for sufficiently large values of $k$ the bound is much worse than the bound given by Theorem~\ref{thm:nik}. On the other hand, Theorem~\ref{thm:ck} also easily shows that $c_k > \frac{1}{k}$ for $k\ge 6$.

\section{Proofs of Theorems~\ref{thm:c4} and \ref{thm:ck} }

Our improved lower bounds are derived from constructions of closed blowups of explicit graphs. Recall that for an integer $t \ge 1$, the closed blowup $G^{[t]}$ of a graph $G$ is the graph obtained by replacing each vertex of $G$ with a $t$-clique and replacing each edge in $G$ with a complete bipartite graph $K_{t, t}$ on the vertices of the $t$-cliques. The eigenvalues of the closed blowup $G^{[t]}$ are $t\lambda_1 + t - 1$, $t\lambda_2 + t - 1$, \ldots, $t\lambda_n+t-1$, along with $(t-1)n$ additional $-1$s, where $\lambda_1 \ge \lambda_2 \ge \cdots \ge \lambda_n$ are the eigenvalues of $G$~\cite[Proposition 5.4]{Nik1}. 

\begin{proof}[Proof of Theorem~\ref{thm:c4}]
Let $G$ be the icosahedral graph. $G$ is a graph on $12$ vertices with spectrum $5^1(\sqrt{5})^3(-1)^5(-\sqrt{5})^3$\cite{Cve}. Therefore, the closed blowups of $G$ satisfy $\lambda_4(G^{[t]}) = t\sqrt{5} + t - 1$, so
\[c_4 \ge \sup_t \frac{\lambda_4(G^{[t]})}{12t} = \sup_t \frac{t\sqrt{5}+t - 1}{12t} = \frac{1+\sqrt{5}}{12}.\]
\end{proof}

\begin{proof}[Proof of Theorem~\ref{thm:ck}]
The Johnson graphs $J(k, 2)$ for $k\ge 4$ have $k$th largest eigenvalue $k-4$ (see \cite[Theorem 6.3.2]{GM}, for example, for the complete spectrum of Johnson graphs). Therefore, the closed blowups $J(k, 2)^{[t]}$ satisfy $\lambda_k(J(k, 2)^{[t]}) = t(k-4) + t - 1$, so 
\[c_k \ge \sup_t \frac{t(k-4) + t - 1}{t\binom{k}{2}} = \frac{2(k-3)}{k(k-1)}.\]
\end{proof}

\section{Concluding remarks}
Perhaps the most immediate open question stemming from the work presented here is to decide if $c_3 > \frac13$. We have been unable to find a construction of a graph $G$ with $\lambda_3 > \frac{n}{3}$. Besides the construction $3K_{\frac{n}{3}}$ mentioned in the beginning of the paper, other examples of graphs with $\lim_{n\rightarrow\infty}\frac{\lambda_3(G)}{n} = \frac13$ include the closed blowups of the $6$-cycle. 

One could also attempt to find better constructions which improve the lower bound on $c_k$ for other values of $k$. As an aid to researchers who might be interested in studying this question further, we conclude with a table of the best lower bound constructions that we know for small values of $k$. In all cases, the construction is a closed blowup of the graph or graphs listed. 

\section*{Acknowledgements}
I thank Clive Elphick for stimulating discussions which inspired this research. I thank the anonymous referee for helpful comments and in particular for making some improvements to Table 1. 

\begin{table}[h!]
\centering
\begin{tabular}{|c|c|c|}\hline
$k$ & $c_k \ge $ & Graph\\
\hline
4 & $\frac{1+\sqrt{5}}{12} \approx 0.26967$ & Icosahedral Graph\\
\hline
5 & $\frac29 \approx 0.2222$ & Paley graph on 9 vertices~\cite{Nik1}\\
\hline
6 & $\frac{1}{5} = 0.2 $ & Petersen graph~\cite{Nik1}, $J(6, 2)$, $J(6, 3)$, Line graph of Petersen graph\\
\hline
7 & $\frac{4}{21} \approx 0.190476$ & $J(7, 2)$\\
\hline
8 & $\frac{5}{28} \approx 0.178571$ & $J(8, 2)$, Gosset graph\\
\hline
9 & $\frac{1}{6} \approx 0.1666$ & $J(9, 2)$\\
\hline
10 & $\frac{7}{45} \approx 0.1555$ & $J(10, 2)$\\
\hline
11 & $\frac{8}{55} \approx 0.14545$ & $J(11, 2)$\\
\hline
12 & $\frac{3}{22} \approx 0.13636$ & $J(12, 2)$\\
\hline
13 & $\frac{5}{39} \approx 0.128205$ & $J(13, 2)$\\
\hline
14 & $\frac{11}{91} \approx 0.1208791$ & $J(14, 2)$\\
\hline
15 & $\frac{4}{35} \approx 0.1142857$ & $J(15, 2)$\\
\hline
16 & $\frac{13}{120} \approx 0.108333$ & $J(16, 2)$\\
\hline
17 & $\frac{2}{19} \approx 0.10526$ & $\text{srg}(57, 24, 11, 9)$\\
\hline
18 & $\frac{2}{19} \approx 0.10526$ & $\text{srg}(57, 24, 11, 9)$\\
\hline
19 & $\frac{2}{19} \approx 0.10526$ & $\text{srg}(57, 24, 11, 9)$\\
\hline
20 & $\frac{13}{125} = 0.104$ & $\text{srg}(125, 72, 45, 36)$\\
\hline
21 & $\frac{13}{125} = 0.104 $ & $\text{srg}(125, 72, 45, 36)$\\
\hline
22 & $\frac{13}{126} \approx 0.10317$ & $\text{srg}(126, 60, 33, 24)$\\
\hline
23 & $\frac{25}{243} \approx 0.10288$ & $\text{srg}(243, 132, 81, 60)$\\
\hline
24 & $\frac{56}{552} \approx 0.101449$ & Taylor graph from Conway group $\text{Co}_3$\\
\hline
\end{tabular}
\caption{Lower bounds for $c_k$}
\label{cktable}
\end{table}

\end{document}